\documentclass[11pt,a4paper,reqno]{amsart}
\usepackage[all]{xy}
\usepackage{amsmath,amsfonts,amssymb,amscd}
\usepackage{soul,color}
\usepackage{xcolor}
\usepackage{graphicx}
\newcommand{\Bmu}{\mbox{$\raisebox{-0.59ex}
  {$l$}\hspace{-0.18em}\mu\hspace{-0.88em}\raisebox{-0.98ex}{\scalebox{2}
  {$\color{white}.$}}\hspace{-0.416em}\raisebox{+0.88ex}
  {$\color{white}.$}\hspace{0.46em}$}{}}

\theoremstyle{plain}
\newtheorem{theorem}{Theorem}
\newtheorem{mytheorem}{Theorem}[subsection]
\newtheorem{corollary}[mytheorem]{Corollary}
\newtheorem{lemma}[mytheorem]{Lemma}
\newtheorem{proposition}[mytheorem]{Proposition}
\newtheorem{remark}[mytheorem]{Remark}

\font\cyr wncyr10 at 11pt \def\Sha{\hbox{\cyr X}}

\def\tF{\textsc{F}}
\def\tV{\textsc{V}}

\def\CC{{\mathbb C}}

\def\FF{{\mathbb F}}

\def\QQ{{\mathbb Q}}

\def\ZZ{{\mathbb Z}}

\def\C{{\mathcal C}}

\def\E{{\mathcal E}}
\def\F{{\mathcal F}}
\def\cG{{\mathcal G}}
\def\cH{{\mathcal H}}

\def\K{{\mathcal K}}
\def\cL{{\mathcal L}}

\def\cO{{\mathcal O}}
\def\cP{{\mathcal P}}

\def\V{{\mathcal V}}

\newcommand{\fC}{\mathfrak{C}}

\newcommand{\fM}{\mathfrak{M}}

\newcommand{\fV}{\mathfrak{V}}
\newcommand{\fW}{\mathfrak{W}}

\newcommand{\fp}{\mathfrak{p}}

\newcommand{\Kpinf}{{K}_{\infty}^{(p)}}

\newcommand{\La}{\operatorname{\Lambda}}

\newcommand{\bo}{\mathbf}

\newcommand{\lr}{{\longrightarrow\;}}

\def\+{{\dagger}}
\DeclareMathOperator{\Ker}{Ker}

\DeclareMathOperator{\coh}{H}
\DeclareMathOperator{\Frob}{Frob}
\DeclareMathOperator{\Fr}{Fr}

\DeclareMathOperator{\Gal}{Gal}

\DeclareMathOperator{\res}{res}

\DeclareMathOperator{\Sel}{Sel}

\DeclareMathOperator{\Hom}{Hom}

\DeclareMathOperator{\Spec}{Spec}

\definecolor{purple}{rgb}{0.7,0,1}

\begin{document}
\title[ ] {The Iwasawa Main conjecture of constant ordinary\\ abelian varieties over function fields}

\author[Lai] {King Fai Lai}
\address{School of Mathematical Sciences\\
Capital Normal University\\
Beijing 100048, China}
\email{kinglaihonkon@gmail.com}

\author[Longhi] {Ignazio Longhi}
\address{Department of Mathematical Sciences\\
Xi'an Jiaotong-Liverpool University\\
No.111 Ren'ai Road, Suzhou Dushu Lake Higher Education Town, Suzhou Industrial Park, Jiangsu, China}
\email{Ignazio.Longhi@xjtlu.edu.cn}

\author[Tan]{Ki-Seng Tan}
\address{Department of Mathematics\\
National Taiwan University\\
Taipei 10764, Taiwan}
\email{tan@math.ntu.edu.tw}

\author[Trihan]{Fabien Trihan}
\address{Department of Information and Communication Sciences\\
Faculty of Science and Technology\\
Sophia University, Tokyo, Japan}
\email{fabientrihan@gmail.com}



\begin{abstract} We study a geometric analogue of the Iwasawa Main Conjecture for constant ordinary abelian varieties over $\ZZ_p^d$-extensions of function fields  ramifying at a finite set of places.
\end{abstract}

\maketitle


\section{Introduction}  \label{sec:intro}

We prove in this paper the Iwasawa Main Conjecture for constant ordinary abelian varieties over
function fields.
Fix a prime number $p$ ($p=2$ is allowed). Let $\FF$ be a finite field of characteristic $p$ and let
$A$ be an ordinary abelian variety defined over $\FF$.
Let $K$ be a global field having $\FF$ as its constant field.
Let $L$ be a $\ZZ_p^d$-extension of $K$ ($d\geq 1$), unramified outside a finite set of places, with $\Gamma:=\Gal(L/K)$,
and write $\La$ for the Iwasawa algebra $\ZZ_p[[\Gamma]]$.
Let $X_p(A/L)$ denote the Pontryagin dual of the Selmer group $\Sel_{p^\infty}(A/L)$.

Note that $X_p(A/L)$ is finitely generated over $\La$, hence we can define the characteristic ideal $\chi\big(X_p(A/L)\big)$,
which is a principal ideal of $\Lambda$.
For $\omega$ a continuous character of $\Gamma$, let $L(A,\omega,s)$ denote the twisted Hasse-Weil $L$-function of $A$.
The following theorem summarizes the main results of this paper.

\begin{theorem} \label{t:summary} There exists a ``$p$-adic $L$-function'' $\mathcal{L}_{A,L}\in \La$ such that for any continuous character
$\omega\colon\Gamma\rightarrow\CC^\times$,
\begin{equation}\label{e:imc2interpolate}
\omega(\mathcal{L}_{A,L})=*_{A,\,L,\, \omega}\cdot L(A,\omega,1)
\end{equation}
for an explicit  factor $*_{A,\,L,\,\omega}$. Furthermore, we have the equality of ideals
\begin{equation}\label{e:imc3characteristic}
(\mathcal{L}_{A,L})= \chi\big(X_p(A/L)\big).
\end{equation}

\end{theorem}

The interpolation formula \eqref{e:imc2interpolate} as well as the explicit expression of $*_{A,\,L,\, \omega}$ are
the content of Theorem \ref{t:interpo}, while \eqref{e:imc3characteristic} is proven in  Theorem \ref{t:imc3constant}.
The $p$-adic $L$-function $\mathcal{L}_{A,L}\in \La$ is defined to be a product of twists of the Stickelberger element associated to $L/K$
(see \eqref{e:plfunction}). By this, the interpolation formula \eqref{e:imc2interpolate} can be straightforwardly deduced.

To prove \eqref{e:imc3characteristic}, in \S\ref{s:frobact} we decompose $X_p(A/L)$ into the Frobenius and the Verschiebung parts.
It turns out that when $L$ contains the arithmetic $\ZZ_p$-extension the Frobenius part is a direct sum of twisted eigenspaces of
class group (see Proposition \ref{p:cs}), and hence its characteristic ideal can be determined by applying the Iwasawa Main Theorem of class groups \cite{crw87}. To deal with the Verschiebung part, we make use of the algebraic functional equation proved in \cite{LLTT1}.
These together lead to the proof of the theorem under the condition that $L$ contains the arithmetic $\ZZ_p$-extension.
If the condition is not satisfied, we first adjoin $L$ with the arithmetic $\ZZ_p$-extension to form a $\ZZ_p^{d+1}$-extension $L_{ar}/K$
so that the theorem holds over $L_{ar}$, then we apply the specialization formula of \cite{tan10b} to complete the proof.

Let us mention also that the Iwasawa Main Conjecture for semistable abelian varieties over the arithmetic $\ZZ_p$-extension was proven in \cite{LLTT2}.
In \S\ref{su:2dreams}, we check that our result is compatible with the latter one.



\begin{subsubsection}*{Acknowledgements} The fourth author has been supported by EPSRC. He would like also to express his gratitude to Takeshi Saito for his hospitality at the University of Tokyo where part of this work has been written. Authors 2, 3 and 4 thank Centre de Recerca Matem\`atica for hospitality while working on part of this paper.
Authors 1, 2 and 3 have been partially supported by the National Science Council of Taiwan, grants NSC98-2115-M-110-008-MY2, NSC100-2811-M-002-079 and NSC99-2115-M-002-002-MY3 respectively. Finally, it is our pleasure to thank NCTS/TPE for supporting a number of meetings of the authors in National Taiwan University. \end{subsubsection}

\begin{section}{Settings and preliminary results on Iwasawa modules} \label{s:setting}

In this section we set notation and recall some facts about Iwasawa algebras and modules.

\begin{subsection}{Basic notation}\label{su:bn}
Let $K^a$ be a fixed algebraic closure of $K$. For each positive integer $m$, let $K^{(1/p^m)}\subset K^a$ denote the unique purely inseparable intermediate extension of degree $p^m$ over $K$, and let $K^{(1/p^{\infty})}:=\bigcup_{m=1}^{\infty} K^{(1/p^m)}$. Let ${\bar K}^{(1/p^m)}\subset K^a$ denote the separable closure of $K^{(1/p^m)}$.
Also, let $\Frob_{p^m}\colon x\mapsto x^{p^m}$ denote the Frobenius substitution: it induces a morphism $Frob_{p^m}\colon K^{(1/p^m)}\longrightarrow K$. On Galois groups, $Frob_{p^m}$ gives rise to an identification
$$\Gal({\bar K}/K)=\Gal({\bar K}^{(1/p^m)}/K^{(1/p^m)}).$$

Let $q$ denote the cardinality of $\FF$. 
Let $\Kpinf:=K\FF_{q^{p^\infty}}$ denote the arithmetic $\ZZ_p$-extension, which is the unique everywhere unramified $\ZZ_p$-extension of $K$.

Let $S$ denote the ramification locus of $L/K$, which is assumed to be a finite set.
As usual, for $v$ a place of $K$, we will write $K_v$ for the completion of $K$ at $v$.

Put
$\Gamma_n:=\Gamma/\Gamma^{p^n}\simeq (\ZZ_p/p^n\ZZ_p)^d$ so that $\La=\varprojlim_n \ZZ_p[\Gamma_n]$. Let $K_n$ denote the $n$th layer of $L/K$.
Hence $\Gamma_n=\Gal(K_n/K)$ and $\Gal(L/K_n)=\Gamma^{p^n}=:\Gamma^{(n)}$.

Let $\cO$ denote the ring of integers of some finite extension of $\QQ_p$.

\begin{subsubsection}{The Selmer groups}
We view $A$ as a sheaf on the flat topology of $K$ and let $A_{p^n}$ denote the kernel of the multiplication by $p^n$ on $A$.
For any finite extension $F/K$, the $p^n$-Selmer group $\Sel_{p^n}(A/K)$ is defined to be the kernel of the composition
\begin{equation*}\label{e:nselmer} \begin{CD} \coh^1_{\mathrm{fl}}(F,A_{p^n}) @>{i^*}>> \coh^1_{\mathrm{fl}}(F,A)@>{loc_F}>> \bigoplus_v \coh^1_{\mathrm{fl}}(F_v,A)\,, \end{CD}\end{equation*}
where $\coh_{\mathrm{fl}}^\bullet$ denotes the flat cohomology and $loc_F$ is the localization map to the direct sum of local cohomology groups over all places of $F$. Taking the direct limit as $n\rightarrow\infty$, we set
\begin{equation*}\label{e:selmer} \Sel_{p^\infty}(A/F):=\Ker\big(\coh^1_{\mathrm{fl}}(F,A_{p^\infty})\lr \bigoplus_{\text{all}\ v}\coh^1_{\mathrm{fl}}(F_v,A)\big).\end{equation*}
Define, by taking direct limit over all finite intermediate extensions,
\begin{equation*}\label{e:selmer} \Sel_{p^\infty}(A/L):=\varinjlim_{F} \Sel_{p^\infty}(A/F). \end{equation*}

\end{subsubsection}
\end{subsection}

\begin{subsection}{Iwasawa algebras}\label{su:iw}
Even if our main interest lies in $\Lambda$,
we are going to need Iwasawa algebras $\La_{\cO}(\Gamma'):=\cO[[\Gamma']]$ for topologically finitely generated abelian $p$-adic Lie groups $\Gamma'$. It has the usual topology as inverse limit of $\cO[\Gamma'/\Delta]$
where $\Delta$ runs through all finite index subgroups and as a topological space $\cO[\Gamma'/\Delta]$ is just a finite product of copies of $\cO$.
In the following, by $\La$- or $\La_\cO(\Gamma')$-module we will always mean a topological one, with continuous action of the scalar ring.

\subsubsection{} \label{ss:iwH1} Any continuous group homomorphism $\Gamma'\rightarrow\La_\cO(\Gamma')^{\times}$ extends by linearity to a
continuous ring homomorphism $\La_\cO(\Gamma')\rightarrow\La_\cO(\Gamma')$. We shall make use of the following:
\begin{enumerate}
\item[(H1)] The inversion $\Gamma'\rightarrow \Gamma'$, $\gamma\mapsto \gamma^{-1}$, gives rise to the isomorphism
$$\cdot^\sharp\colon\La_\cO(\Gamma')\,\lr \La_\cO(\Gamma')\,,$$
sending an element $\lambda$ to $\lambda^\sharp$. In particular, if $\langle\;,\;\rangle$ is a $\Gamma$-equivariant pairing between $\La$-modules, in the sense that $\langle \gamma a,\gamma b \rangle=\langle a, b\rangle$, for all $\gamma\in\Gamma$, then
\begin{equation} \label{e:twistpairing} \langle \lambda\cdot a,b\rangle=\langle a,\lambda^\sharp\cdot b\rangle,\;\;\;\text{for every}\; \lambda\in\La.
\end{equation}

\item[(H2)] Every continuous character $\phi\colon{\Gamma'}\rightarrow \cO^{\times}$ induces a homomorphism $\phi^*\colon\Gamma'\rightarrow\La_\cO(\Gamma')^{\times}$, $\phi^*(\gamma):= \phi(\gamma)^{-1}\cdot\gamma$, $\gamma\in\Gamma'$,
    and this extends to the ring homomorphism
$$\phi^*\colon\La_\cO(\Gamma')\,\longrightarrow\La_\cO(\Gamma').$$
Since on $\Gamma'$ the composition $\phi^*\circ (1/\phi)^*$ is the identity map, $\phi^*$ is an isomorphism on $\La_\cO(\Gamma')$.
\end{enumerate}

\end{subsection}


\begin{subsection}{Iwasawa modules}\label{su:ba} Assume that $\Gamma'$ is isomorphic to $\ZZ_p^n$ for some $n$.
The choice of a $\ZZ_p$-basis
$\{ \gamma_i\}$
for $\Gamma'=\bigoplus_{i=1}^n\gamma_i^{\ZZ_p}\simeq\ZZ_p^n$ yields an isomorphism $\La_{\cO}(\Gamma')\simeq\cO[[T_1,...,T_n]]$ where $T_i:=\gamma_i-1$. It follows that the Iwasawa algebra $\La_{\cO}(\Gamma')$ is indeed a unique factorization domain.
By definition a $\La_{\cO}(\Gamma')$-module $M$ is pseudo-null if and only if no height one prime ideal contains its annihilator (i.e., if for any height one prime $\fp$ the localization $M_\fp=0$). The proof of the following lemma is omitted. A comprehensive reference is \cite[\S4]{bou65}.

\begin{lemma}\label{l:psn}
A finitely generated $\La_{\cO}(\Gamma')$-module $M$ is pseudo-null if and only if there exist relatively prime $f_1,...,f_k\in\La_{\cO}$, $k\geq2$, so that $f_iM=0$ for every $i$.
\end{lemma}


A pseudo-isomorphism is a homomorphism with pseudo-null kernel and cokernel. We will write $M\sim N$ to mean that there exists a pseudo-isomorphism from $M$ to $N$.

\subsubsection{}\label{sss:krul} Suppose $\Gamma'\simeq\ZZ_p^n$ and $M$ is a finitely generated torsion $\La_{\cO}(\Gamma')$-module. By the general theory of modules over a Krull domain, there is a pseudo-isomorphism
\begin{equation*} \label{e:structureM} \Phi\colon\bigoplus_{i=1}^m \La_{\cO}(\Gamma')/\xi_i^{r_i}\La_{\cO}(\Gamma')\longrightarrow M,\end{equation*}
where each $\xi_i$ is irreducible.
Since no element in $\La_{\cO}(\Gamma')/\xi_i^{r_i}\La_{\cO}(\Gamma')$ is annihilated by two relative prime elements of $\La_{\cO}(\Gamma')$,
by Lemma \ref{l:psn}, $\Phi$ is injective. Define
\begin{equation}\label{e:[m]}
[M]:=\bigoplus_{i=1}^m \La_{\cO}(\Gamma')/\xi_i^{r_i}\La_{\cO}(\Gamma'),
\end{equation}
which is independent of $\Phi$. The characteristic ideal of $M$ is defined to be
$$\chi_{\cO,\Gamma'}(M):=(\prod_{i=1}^m (\xi_i^{r_i}))\subset\La_{\cO}.$$
If $M$ is non-torsion define $\chi_{\cO,\Gamma'}(M)=0$. If $\cO=\ZZ_p$ and $\Gamma'=\Gamma$, we write $\chi:=\chi_{\cO,\Gamma'}$.

\subsection{Twists} \label{su:twistmodule}
Let $M$ be a $\La_{\cO}(\Gamma')$-module. Any endomorphism $\alpha\colon\La_{\cO}(\Gamma')\to\La_{\cO}(\Gamma')$ defines a twisted $\La_{\cO}(\Gamma')$-module
$\La_{\cO}(\Gamma')\,{}_\alpha\!\otimes_{\La_{\cO}(\Gamma')}M\,,$ where the action on the copy of $\La_{\cO}(\Gamma')$ on the left is via $\alpha$ (i.e., we have $(\alpha(\lambda)\mu)\otimes m=\mu\otimes\lambda m$ for $\lambda$, $\mu\in\La_{\cO}(\Gamma')$ and $m\in M$) and the module structure is given by
\begin{equation} \label{e:twistaction} \lambda\cdot(\mu\otimes m) := (\lambda\mu)\otimes m \end{equation}
(where $\lambda\mu$ is the product in $\La_{\cO}(\Gamma')$ ).
If moreover $\alpha$ is an automorphism, $\La_{\cO}(\Gamma')\,{}_\alpha\!\otimes_{\La_{\cO}(\Gamma')}M$ can be identified with $M$ with the $\La_\cO(\Gamma')$-action twisted by $\alpha^{-1}$, since in this case \eqref{e:twistaction} becomes
\begin{equation} \label{e:twistaction2} \lambda\cdot(1\otimes m)=1\otimes\alpha^{-1}(\lambda)m \,.\end{equation}

\begin{lemma}\label{l:phi[]chi} Let $\alpha$ be an automorphism of $\La_\cO(\Gamma')\simeq\cO[[T_1,...,T_n]]$. Suppose $M$ is a finitely generated torsion $\La_{\cO}(\Gamma')$-module with
$$[M]=\bigoplus_{i=1}^m \La_{\cO}(\Gamma')/\xi_i^{r_i}\La_{\cO}(\Gamma')\,.$$
Then
$$[\La_{\cO}(\Gamma')\,{}_\alpha\!\otimes_{\La_{\cO}(\Gamma')}M]=\La_{\cO}(\Gamma')\,{}_\alpha\!\otimes_{\La_{\cO}(\Gamma')}[M]=\bigoplus_{i=1}^m \La_{\cO}(\Gamma')/\alpha(\xi_i)^{r_i}\La_{\cO}(\Gamma'),$$
and hence
$$\chi_{\cO,\Gamma'}\big(\La_{\cO}(\Gamma')\,{}_\alpha\!\otimes_{\La_{\cO}(\Gamma')}M\big)=\alpha(\chi_{\cO,\Gamma'}(M)).$$
\end{lemma}

\begin{proof}
It is immediate from \eqref{e:twistaction2} that if a pseudo-null $N$ is annihilated by coprime $f_1$, $f_2$, then $\La_{\cO}(\Gamma')\,{}_\alpha\!\otimes N$ is annihilated by coprime $\alpha(f_1)$, $\alpha(f_2)$, whence pseudo-null. Thus the functor $\La_{\cO}(\Gamma')\,{}_\alpha\!\otimes-$ preserves pseudo-isomorphisms. Since it also commutes with direct sums, we are reduced to check the equality
$$\La_{\cO}(\Gamma')\,{}_\alpha\!\otimes\big(\La_{\cO}(\Gamma')/\xi_i^{r_i}\La_{\cO}(\Gamma')\big)=\La_{\cO}(\Gamma')/\alpha(\xi_i)^{r_i}\La_{\cO}(\Gamma')\,,$$
which is obvious by exactness of $\La_{\cO}(\Gamma')\,{}_\alpha\!\otimes-$.
\end{proof}

We apply the above with the isomorphisms considered in \S\ref{ss:iwH1} and write
$$M^\sharp:=\La_{\cO}(\Gamma')\,{}_\sharp\!\otimes_{\La_{\cO}(\Gamma')}M$$
and, for $\phi$ as in (H2),
\begin{equation} \label{e:twistbyphi} M(\phi):=\La_{\cO}(\Gamma')\,{}_{\phi^*}\!\otimes_{\La_{\cO}(\Gamma')}M. \end{equation}
Since $\cdot^\sharp$ is an involution, \eqref{e:twistaction2} shows that the action of $\La_{\cO}(\Gamma')$ becomes $\lambda\cdot m=\lambda^\sharp m$.
As for $\phi^*$, note that, if we endow $\cO$ with the trivial action of $\Gamma'$, then the $\La_{\cO}$-module
$\cO(\phi)$ can be viewed as the free rank one $\cO$-module with the action of $\Gamma'$ through multiplication by $\phi$, in the sense that
$$\gamma \cdot a=\phi(\gamma) a \text{ for all } \gamma\in\Gamma', a\in \cO(\phi)\,.$$
Then for a $\La_{\cO}$-module $M$ we have
$$M(\phi)=\cO(\phi)\otimes_{\cO}M,$$
where $\Gamma'$ acts by
$$\gamma\cdot(a\otimes x):=(\gamma\cdot a)\otimes (\gamma\cdot x)=\phi(\gamma)\cdot (a\otimes \gamma x)\,.$$

\subsection{Some more notation} In order to lighten the notation, for an $\cO$-module $\fM$ and an $\cO$-algebra $R$, we will often use the shortening
$$R\fM:=R\otimes_{\cO}\fM\,.$$

The Pontryagin dual of an abelian group $G$ will be denoted $G^\vee$. Since we are going to deal mostly with finite $p$-groups and their inductive and projective limits, we generally won't distinguish between the Pontryagin dual and the set of continuous homomorphisms into the group of roots of unity $\Bmu_{p^{\infty}}:=\cup_m\Bmu_{p^{m}}$. Note that we shall usually think of $\Bmu_{p^{\infty}}$ as a subset of $\bar\QQ_p$ (hence with the discrete topology), so that for a $\La$-module $M$ homomorphisms in $M^\vee$ will often take value in $\bar\QQ_p$.

We shall denote the $\psi$-part of a $G$-module $M$ (for $G$ a group and $\psi\in G^\vee$) by
\begin{equation} \label{e:psipart} M^{(\psi)}:=\{x\in M\;\mid\; g\cdot x=\psi(g)x\;\text{for all}\;g\in G\}. \end{equation}

\end{subsection}
\end{section}

\begin{section}{The Frobenius Action} \label{s:frobact}

\begin{subsection}{Absolute and relative Frobenius}\label{su:fm}
Let $\cG$ be a commutative finite group scheme over $K$ and denote $\cG_0$ its connected component and $\cG^{\mathrm{\acute et}}$ its maximal \'etale factor. Then we have the connected-\'etale sequence
\begin{equation} \label{e:connetal} 0 \lr \cG_0\, \lr \cG\, \lr \cG^{\mathrm{\acute et}}\, \lr 0. \end{equation}
We shall view $\cG$, $\cG_0$ and $\cG^{\mathrm{\acute et}}$ as sheaves on the flat topology of $K$, so that we can consider their cohomology groups. The base change $K\lr K^{(1/p^m)}$ induces the restriction map:
$$\res_m\colon\coh^1_{\mathrm{fl}}(K,\cG)\,\lr \coh^1_{\mathrm{fl}}(K^{(1/p^m)},\cG).$$

\begin{lemma}\label{l:image}
The image $\res_m(\coh^1_{\mathrm{fl}}(K,\cG))$ injects into $\coh^1_{\mathrm{fl}}(K^{(1/p^m)},\cG^{\mathrm{\acute et}})$ for any $m$ such that $p^m\geq\dim_K K[\cG_0]$.
\end{lemma}

\begin{proof}
Suppose $\xi\in \coh^1_{\mathrm{fl}}(K,\cG_0)$. Since $\coh^1_{\mathrm{fl}}(K,\cG_0)$ can be computed as the \v{C}ech cohomology (see \cite[Proposition III.6.1]{mil86})
$$ \breve{\coh}^1(K,\cG_0):=\breve{\coh}^1(K^a/K,\cG_0),$$
by \cite[Corollary III.4.7]{mil80} the class $\xi$ corresponds to a principal homogeneous space $\cP$ for $\cG_0$ such that $\cP(F)$, the set of $F$-points of $\cP$, is non-empty for some finite extension $F/K$. The field $F$ can be taken as the residue field of the coordinate
ring $K[\cP]$ modulo a maximal ideal.
As $\cP(K^a)=\cG_0(K^a)=\{id\}$, the ring $K[\cP]$ must contain exactly one maximal ideal and the residue field must be a purely inseparable extension over $K$. This shows that $F$ can be taken as $K^{(1/p^n)}$ for some $n$, and hence $\cG_0\simeq\cP$ over $K^{(1/p^n)}$. Also, since $K[\cP]$ is a finite algebra over $K$ with dimension the same as $\dim_K K[\cG_0]:=f$, the degree of $F/K$ is at most $f$. Therefore, $\coh^1_{\mathrm{fl}}(K^{(1/p^m)},\cG_0)=0$ for $p^m\geq f$.

To conclude, it is enough to take the cohomology of the sequence \eqref{e:connetal}.
\end{proof}

\begin{remark} {\em If $\cG$ is defined over a perfect field then \eqref{e:connetal} splits and the decomposition
$$\cG=\cG_0\times \cG^{\mathrm{\acute et}}$$
holds (see \cite[\S6.8]{wat79}). In particular we have
\begin{equation}\label{e:dec}
\coh^1_{\mathrm{fl}}(K,\cG)=\coh^1_{\mathrm{fl}}(K,\cG_0)\times \coh^1_{\mathrm{fl}}(K,\cG^{\mathrm{\acute et}}).
\end{equation} }
\end{remark}

\subsubsection{Application to abelian varieties} \label{sss:applabvar} Assume that $B/K$ is an abelian variety. Let $B_{p^n}^{\mathrm{\acute et}}$ denote the maximal \'etale factor of $B_{p^n}$.
Lemma \ref{l:image} immediately yields the following.

\begin{corollary}\label{c:image2}
For each $n$, there exists some $m$ so that $\res_m(\coh^1_{\mathrm{fl}}(K,B_{p^n}))$ injects into $\coh^1_{\mathrm{fl}}(K^{(1/p^m)},B_{p^n}^{\mathrm{\acute et}})$.
\end{corollary}

By definition, the abelian variety $B^{(p^m)}:=B\times_K K$ is the base change over the absolute Frobenius $\Frob_{p^m}\colon K\rightarrow K$; as a sheaf on the flat topology of $K$, $B^{(p^m)}/K$ is identified with the inverse image $\Frob_{p^m}^*B$ (\cite[Remark II.3.1.(d)]{mil80}). As we noticed before, the absolute Frobenius factors through
\[ \begin{CD} K @>>> K^{(1/p^m)} @>{Frob_{p^m}}>> K. \end{CD} \]
Therefore we get a commutative diagram
\begin{equation} \begin{CD} \label{d:frob} \coh^1_{\mathrm{fl}}(K,B_{p^n})  @>\res_m>> \coh^1_{\mathrm{fl}}( K^{(1/p^m)},B_{p^n}^{\mathrm{\acute et}}) \\
@|   @VV{Frob_{p^m}^*}V  \\
\coh^1_{\mathrm{fl}}(K,B_{p^n}) @>{\Frob_{p^m}^*}>> \coh^1_{\mathrm{fl}}(K,(B^{\mathrm{\acute et}}_{p^n})^{(p^m)}).
\end{CD} \end{equation}

Let $\tF^{(m)}_B\colon B\rightarrow B^{(p^m)}$ denote the relative Frobenius homomorphism: it is the map induced by the commutative diagram
\[\begin{CD} B  @>{\Frob_{p^m}}>> B \\
@VVV    @VVV  \\
K @>{\Frob_{p^m}}>> K. \end{CD}\]

\begin{lemma} \label{l:absrelfrob} As maps of sheaves on the flat topology of $K$, $(\tF^{(m)}_B)_*$ and $\Frob_{p^m}^*$ coincide.
\end{lemma}

\begin{proof} Denote by $\pi$ the projection $B^{(p^m)}=B\times_K K\rightarrow B$. Let $s$ be a section of $B/K$: by definition $(\tF^{(m)}_B)_*(s)=\tF^{(m)}_B\circ s$, while $\Frob_{p^m}^*(s)$ is the unique section such that $\pi\circ\Frob_{p^m}^*(s)=s\circ\Frob_{p^m}$. Since $\pi\circ\tF^{(m)}_B=\Frob_{p^m}$, the claim is reduced to check that $\Frob_{p^m}\circ s=s\circ\Frob_{p^m}$ and this follows from the trivial observation that the map $x\mapsto x^{p^m}$ commutes with every homomorphism of $\FF_p$-algebras.
\end{proof}

Together with Corollary \ref{c:image2} and diagram \eqref{d:frob}, Lemma \ref{l:absrelfrob} implies the following corollary in an obvious way.

\begin{corollary}\label{c:image1}
For each $n$, there exists some $m$ so that $(\tF^{(m)}_B)_*(\coh^1_{\mathrm{fl}}(K,B_{p^n}))$, and hence $(\tF^{(m)}_B)_*(\Sel_{p^n}(B/K))$, is contained in $\coh^1_{\mathrm{fl}}(K,(B_{p^n}^{\mathrm{\acute et}})^{(p^m)})$.
\end{corollary}

\end{subsection}


\begin{subsection}{The Frobenius decomposition}\label{su:fa} In the case of our constant abelian variety $A$, the map $\tF_{A,\,q}:=\tF_A^{(a)}$ (with $a$ such that $q=|\FF|=p^a$) is an endomorphism and it satisfies $\tF_A^{(am)}=\tF_{A,\,q}^m\,$. In the following we shall generally shorten $\tF_{A,\,q}$ to $\tF_q\,$.

Let $\E=\E_A$ denote the ring of endomorphisms of $A/K$ and write $\ZZ_p\,\E:=\ZZ_p\otimes_{\ZZ}\E$.
\begin{lemma}\label{l:silly} The endomorphism rings $\E_A$ is commutative.
\end{lemma}
\begin{proof}
\cite[Proposition 7.1]{wat69}
\end{proof}
We have
$$[q^m]=\tV_q^{(m)}\circ \tF_q^m$$
for some $\tV_q^{(m)}\in \E$. Write $\tV_q$ for $\tV_q^{(1)}$.
The above lemma implies that
$$[q^m]=[q]^m=\tV_q^m\circ \tF_q^m,$$
and hence $\tV_q^{(m)}=\tV_q^m$.
The following lemma is from \cite[Theorem 7.2]{wat69}.



\begin{lemma}\label{l:unit}
The operator $\tV_q^m+\tF_q^m$ is invertible in $\ZZ_p\,\E$ for all $m\geq1$.
\end{lemma}


\begin{corollary} \label{c:Esplits}
We have the decomposition of $\ZZ_p$-algebra:
$$\ZZ_p\,\E=\F\times \V\,,$$
where
$$\F:=\bigcap_{m\geq0}\tF_q^m(\ZZ_p\,\E) \hspace{17pt}{ and }\hspace{17pt} \V:=\bigcap_{m\geq0}\tV_q^m(\ZZ_p\,\E) \,. $$
\end{corollary}

\begin{proof} The ring $\ZZ_p\,\E$ has a natural topology as the $p$-adic completion of $\E$. Since $\ZZ_ p\,\E$ is compact, it is possible to find subsequences $\tF_q^{m_k}$, $\tV_q^{m_k}$ converging respectively to $\tF_\infty\in \F$, $\tV_\infty\in \V$.
By Lemma \ref{l:unit} and the fact that $(\ZZ_p\,\E)^{\times}$ is closed,
$$\tF_\infty+\tV_\infty=\lim_{m_k\rightarrow\infty}(\tF_q^{m_k}+\tV_q^{m_k})$$
is a unit. To complete the proof, we only need to note that $\tF_\infty\tV_\infty=0$ since $\tF_\infty \tV_\infty\in \tF_q^m\tV_q^m(\ZZ_p\,\E)=q^m\ZZ_p\,\E$ for all $m$.
\end{proof}

By Corollary \ref{c:Esplits}, any $\ZZ_p\,\E$-module $M$ decomposes as $\F M\oplus \V M$. In many cases, a better grasp of the two components can be obtained by the following lemma (and its obvious $\V$-analogue).

\begin{lemma}\label{l:decomposition} Let $W$ be a $\ZZ_p\,\E$-module and $W[q^n]$ its $q^n$-torsion submodule. Then
$$\F W[q^n]=\tF_q^n(W[q^n]).$$
In particular, if $W$ is a $p$-primary abelian group then $\F W=\bigcap_n \tF_q^nW$.
\end{lemma}

\begin{proof}
By Lemma \ref{l:unit}, for any $m\geq n$ and any $r\geq 1$ we have
$$\tF_q^{rm}\ZZ_p\,\E+q^n\ZZ_p\,\E=(\tF_q^m+\tV_q^m)(\tF_q^{rm}\ZZ_p\,\E+q^n\ZZ_p\,\E)=\tF_q^{(r+1)m}\ZZ_p\,\E+q^n\ZZ_p\,\E\,.$$
Hence the image of $\F$ in $\ZZ_p\,\E/q^n\ZZ_p\,\E$ is $\tF_q^n(\ZZ_p\,\E/q^n\ZZ_p\,\E)$.
\end{proof}

Let $\bullet^t:\E_A\longrightarrow \E_{A^t}$ denote the anti-isomorphism sending an endomorphism to its dual, and
extend it to $\ZZ_p\E_A\longrightarrow \ZZ_p\E_{A^t}$.
Since $(\tF_{A,\,q})^t=\tV_{A^t,\,q}$ and $(\tV_{A,\,q})^t=\tF_{A^t,\,q}\,$ ( \cite[Proposition 7.34]{gm13}), we have
\begin{equation}\label{e:fatvt} (\F_A)^t=\V_{A^t}\, \text{ and } (\V_A)^t=\F_{A^t}\,\,. \end{equation}

\begin{proposition}\label{p:sfft}
If $X_p(A/L)$ is a torsion $\La$-module, then
$$\chi(\F_A X_p(A/L))= \chi(\V_{A^t} X_p(A^t/L))^\sharp\,\; \text{and}\,\; \chi(\V_A X_p(A/L))= \chi(\F_{A^t} X_p(A^t/L))^\sharp.$$
\end{proposition}

\begin{proof} By \eqref{e:fatvt} this is just a reformulation of \cite[Theorem 5.2.6]{LLTT1} (note that since $A$ is defined over $\FF$ and ordinary, it has good ordinary reduction at every place of $K$. By \cite[Theorem 3]{Mi68} it is known that $\Sha_{p^\infty}(A/K_n)$ is finite for every $n$. Therefore, the condition of \cite[Theorem 5.2.6]{LLTT1} is satisfied).
\end{proof}

\begin{lemma}\label{l:sst}
If one of the following modules
$$X_p(A/L),\; \F_A X_p(A/L),\;  \V_A X_p(A/L),$$
and
$$X_p(A^t/L),\; \F_{A^t} X_p(A^t/L),\;  \V_{A^t} X_p(A^t/L),$$
is torsion over $\La$, then all of them are torsion over $\La$.
\end{lemma}

\begin{proof}
Applying any sequence of isogenies
$$A\longrightarrow A^t \longrightarrow A \longrightarrow A^t,$$
so that the composition of the last (resp. first) two arrows equals some $[n]$ (resp. $[n']$), we see that if any item of the upper (resp. lower) row is annihilated by a non-zero $\xi\in\La$, then the corresponding item in the lower (resp. upper) row is annihilated by $n'\xi$ (resp. $n\xi$). In view of the decomposition
\begin{equation} \label{e:xfv} X_p(A/L)=\F_A X_p(A/L)\oplus \V_A X_p(A/L)\,, \end{equation}
we only need to show that $\cH_AX_p(A/L)$ is torsion if and only if so is $\K_{A^t}X_p(A^t/L)$, with $\{\cH,\K\}=\{\F,\V\}$.

Suppose $\cH_AX_p(A/L)$ is annihilated by $\xi\in\La$. Then $\xi^\sharp\cdot \cH_A\Sel_{p^\infty}(A/L)=0$. Let $g(T)$ denote the characteristic polynomial of the action on the Tate module of $A(L)[p^\infty]$ by a non-zero element $\gamma\in \Gamma$. Then
$g(\gamma)\cdot A(L)[p^\infty]$ is a finite group of order, say, $p^m$. Then $\delta:=p^mg(\gamma)$ annihilates $A(L)[p^\infty]$, and consequently $\delta\cdot \coh^1(\Gamma^{(n)},A_{p^\infty}(L))=0$, for every $n$.
Because
$$\xi^\sharp\cdot \cH_A\Sel_{p^\infty}(A/K_n)\subset \Ker \big(\Sel_{p^{\infty}}(A/K_n)\lr\Sel_{p^\infty}(A/L)\big)= \coh^1(\Gamma^{(n)},A_{p^\infty}(L)),$$
we have $\delta\xi^\sharp\cdot \cH_A\Sel_{p^\infty}(A/K_n)=0$. This implies that
$\delta\xi^\sharp$ annihilates both $\cH_A\Sha_{p^\infty}(A/K_n)$ and  $\cH_A(\QQ_p/\ZZ_p\otimes A(K_n))$,
by the exact sequence
$$0\longrightarrow \QQ_p/\ZZ_p\otimes A(K_n)  \longrightarrow  \Sel_{p^\infty}(A/K_n) \longrightarrow \Sha_{p^\infty}(A/K_n) \longrightarrow 0.$$
By the Cassels-Tate pairing the identities \eqref{e:twistpairing} and \eqref{e:fatvt} imply $\delta^\sharp\xi\cdot\K_{A^t}\Sha_{p^\infty}(A^t/K_n)=0$.
It remains to show that $\delta^\sharp\xi$ annihilates $\K_{A^t}(\QQ_p/\ZZ_p\otimes A^t(K_n))$,
so that
$(\delta^\sharp\xi)^2$ annihilates $\K_{A^t}\Sel_{p^\infty}(A^t/K_n)$ for all $n$. Hence, it
also annihilates $\K_{A^t}\Sel_{p^\infty}(A^t/L)$. Then by duality, $(\delta\xi^\sharp)^2\cdot \K_{A^t}X_{p^{\infty}}(A^t/L)=0$, as desired.

But since the N$\acute{\text{e}}$ron-Tate height induces a $\Gamma$-equivariant non-degenerate pairing (see \cite[Lemma 5.2.2]{LLTT1} and its proof)
$$\QQ_p\otimes A(K_n)\times \QQ_p\otimes A^t(K_n)\longrightarrow \QQ_p,$$
by \eqref{e:twistpairing} and \eqref{e:fatvt}, the $p$-divisible group $\delta^\sharp\xi\cdot\K_{A^t}(\QQ_p/\ZZ_p\otimes A^t(K_n))$ must be finite,
whence trivial.
\end{proof}

\begin{proposition}\label{p:xffsharp}
If $A$ is an ordinary abelian variety over $\FF$, then
$$\chi(X_p(A/L))=\chi(\F X_p(A/L))\cdot \chi(\F_{A^t} X_p(A^t/L))^\sharp.$$
\end{proposition}

\begin{proof} If $\F X_p(A/L)$ is torsion, then the equality follows from Lemma \ref{l:sst}, Proposition \ref{p:sfft} and \eqref{e:xfv}; otherwise, we get $0=0$. \end{proof}

\end{subsection}


\begin{subsection}{Selmer groups and class groups} \label{su:sgcg}

In this section we reduce the computation of $\F X_p(A/L)$ to that of Galois cohomology groups.

Lemma \ref{l:decomposition} and Corollary \ref{c:image1} imply that $\F\cdot\Sel_{p^n}(A/F)$ is contained in $\coh^1_{\mathrm{fl}}(F,A_{p^n}^{\mathrm{\acute et}})$. The latter can be viewed as Galois cohomology, by \cite[III, Theorem 3.9]{mil80}: since $A_{p^n}^{\mathrm{\acute et}}$ is \'etale, in the following we will not distinguish between the sheaf $A_{p^n}^{\mathrm{\acute et}}$ and the usual $p^n$-torsion subgroup $A_{p^n}^{\mathrm{\acute et}}(\bar K)=A(\bar K)[p^n]$. Also, when considering Galois cohomology, we shall shorten $A(\bar K)[p^n]$ to $A[p^n]$ and $\Gal(\bar F/F)$ to $G_F$. In particular, we shall write
$$A[p^\infty]:=A(\bar K)[p^\infty]=\bigcup A[p^n]\,.$$

Let $\tilde K_n$ denote the compositum $\FF(A[p^n])K$: it is an abelian extension over $K$, such that, for $n\geq 1$, $\Gal(\tilde K_{n+1}/\tilde K_n)$ is an abelian $p$-group annihilated by $p$.
Suppose $\tilde K_n\subset F$: then we can identify $\coh_{\mathrm{Gal}}^1(F,A[p^n])$ with $\Hom(G_F, A[p^n])$. Recall that $\varphi\in \Hom(G_F, A[p^n])$ is said to be unramified at a place $v$ if $\Ker(\varphi)$ contains the inertia group at $v$.
We use \eqref{e:dec} to identify
$\coh_{\mathrm{Gal}}^1(F,A[p^n])=\coh_{\mathrm{\acute et}}^1(F,A_{p^n}^{\mathrm{\acute et}})=\coh_{\mathrm{fl}}^1(F,A_{p^n}^{\mathrm{\acute et}})$
with a subgroup of $\coh_{\mathrm{fl}}^1(F,A_{p^n})$.

\begin{lemma} \label{l:unramified}
Suppose $F/K$ is a finite extension with $\tilde K_n\subset F$. Then an element $\varphi\in \Hom(G_F, A[p^n])=\coh_{\mathrm{Gal}}^1(F,A[p^n])$ is contained in $\Sel_{p^n}(A/F)$ if and only if $\varphi$ is unramified everywhere.
\end{lemma}


\begin{proof} For $v$ a place of $F$, let $F_v^{unr}$ denote an unramified closure of the completion $F_v$.
Consider the diagram (with exact lines)
\begin{equation} \label{e:diagrunr} \begin{CD}
0 @>>> \coh_{\mathrm{Gal}}^1(F_v^{unr}/F_v,A[p^n]) @>>> \coh_{\mathrm{Gal}}^1(F_v,A[p^n]) @>>> \coh_{\mathrm{Gal}}^1(F_v^{unr},A[p^n]) \\
&& @VVV @VVV @VVV \\
0 @>>> \coh_{\mathrm{Gal}}^1(F_v^{unr}/F_v,A) @>>> \coh_{\mathrm{Gal}}^1(F_v,A) @>>> \coh_{\mathrm{Gal}}^1(F_v^{unr},A)\,.
\end{CD}\end{equation}
The last vertical map is injective: for, if $\varphi$ is contained in the kernel, then there exists a point $Q\in A(\bar F_v)$ such that $\varphi(\sigma)=\sigma Q-Q $ for all $\sigma$ in the inertia subgroup. As the latter acts trivially on the residue field, the reduction ${\overline{ \varphi(\sigma)}}$ is zero. But since $A={\bar A}$ (the reduction of $A$ at $v$),
the reduction map induces an isomorphism $A[p^n]\simeq {\bar A}[p^n]$.
Then we get $\varphi(\sigma)=0$, and hence $\varphi=0$. Also, we have $\coh_{\mathrm{Gal}}^1(F_v^{unr}/F_v,A)=0$ by \cite[I, Proposition 3.8]{mil86} (a consequence of Lang's theorem and Hensel's lemma). Therefore, the kernel of the central vertical map in \eqref{e:diagrunr} is exactly $\coh_{\mathrm{Gal}}^1(F_v^{unr}/F_v,A[p^n])$.

Now it is enough to observe that by definition a cohomology class in $\coh_{\mathrm{Gal}}^1(F,A[p^n])$ is unramified at $v$
if its restriction to $\coh_{\mathrm{Gal}}^1(F_v,A[p^n])$ is contained in $\coh_{\mathrm{Gal}}^1(F_v^{unr}/F_v,A[p^n])$.
\end{proof}

Let $\fW_F$ denote the $p$-completion of the divisor class group of $F$. Note that the class group of $F$ consists of degree zero divisor classes: let $\fC_F$ be its Sylow $p$-subgroup. Then the degree map induces the exact sequence
\begin{equation}\label{e:weil} \begin{CD} 0 @>>> \fC_F @>>> \fW_F @>\deg>> \ZZ_p @>>> 0\,. \end{CD} \end{equation}

\begin{lemma}\label{l:s}
Suppose $F/K$ is a finite extension with $\tilde K_n\subset F$. Then
$$\F\cdot\Sel_{p^n}(A/F)=\Hom(\fW_F , A[p^n]).$$
\end{lemma}

\begin{proof} Let $\fV$ denote the group of homomorphisms $\Hom(\fW_F, A[p^n])$. Class field theory identifies $\fW_F$ with $\Gal(F^{unr,p}/F)$, where $F^{unr,p}$ is the maximal everywhere unramified abelian pro-$p$-extension of $F$. Thus, since $\tF_q^m$ induces an isomorphism on $A[p^n]$, Lemma \ref{l:unramified}, Lemma \ref{l:decomposition} and Corollary \ref{c:image1} imply, for $m\gg 0$,
$$\fV={\tF_q^m}_*(\fV)\subset {\tF_q^m}_*(\Sel_{p^n}(A/F))= \F\cdot\Sel_{p^n}(A/F)\subset\fV.$$
\end{proof}

\begin{subsubsection}{The extensions $L_{ar}$ and $\tilde L$} \label{ss:LL} Let $\FF(A[p^{\infty}])\subset\bar\FF$ be the field of definition of $A[p^{\infty}]$. We have $\Gal(\FF(A[p^{\infty}])/\FF)\simeq\ZZ_p\times H$ for some finite cyclic group $H$ of order prime to $p$. Let $L_{ar}$ denote the compositum $LK_{\infty}^{(p)}$, with $\Gamma_{ar}:=\Gal(L_{ar}/K)\simeq\ZZ_p^{d'}$, $d'\geq d$, and let
$${\tilde L}:=L\FF(A[p^{\infty}]),$$
with ${\tilde \Gamma}:=\Gal({\tilde L}/K)$. By a slight abuse of notation we can write
\begin{equation}\label{e:prod2}
{\tilde \Gamma}= \Gamma_{ar} \times H \,.
\end{equation}
By the main theorem of \cite{tan10a}, $X_p(A/{\tilde L})$ is a finitely generated $\Lambda(\Gamma_{ar})$-module, whence so is $\F X_p(A/\tilde L)$. Define
$$\fW_{\tilde L}:=\varprojlim_F\fW_F.$$

\begin{proposition} \label{p:selclassg} We have a canonical isomorphism of $\Gal({\tilde L}/K)$-modules:
$$\F X_p(A/\tilde L)\simeq \fW_{\tilde L}\otimes_{\ZZ_p}A[p^{\infty}]^\vee. $$
\end{proposition}

\begin{proof} Since $\fW_{\tilde L}$ is compact and $A[p^\infty]$ discrete, we have
$$\Hom_{cont}(\fW_{\tilde L} , A[p^n])=\varinjlim_F\Hom(\fW_F , A[p^n])\,,$$
where $F$ runs through all finite degree subextensions of ${\tilde L}/K$ (note that all homomorphisms from a finitely generated $\ZZ_p$-module such as $\fW_F$ into a finite group are continuous), and of course
$$\varinjlim_n\Hom_{cont}(\fW_{\tilde L} , A[p^n])=\Hom_{cont}(\fW_{\tilde L} , A[p^\infty])\,.$$
Hence we get, by Lemma \ref{l:s},
$$\F X_p(A/\tilde L)=\Hom_{cont}(\fW_{\tilde L} , A[p^\infty])^\vee\,.$$
Finally, the map
$$\fW_{\tilde L}\otimes_{\ZZ_p}A[p^\infty]^\vee \lr \Hom_{cont}(\fW_{\tilde L} , A[p^\infty])^\vee\,,$$
which sends $w\otimes Q^*$ to $\varphi\mapsto Q^*(\varphi(w))$, is an isomorphism because $A[p^\infty]^\vee$ is a finitely generated free $\ZZ_p$-module.
\end{proof}

Note that if $F$ is a finite extension of $K$ with constant field of cardinality $q^n$, then we have the commutative diagram  \[ \begin{CD}
\fW_F  @>{\deg_F}>> \ZZ_p \\
@VVV   @VV{n}V  \\
\fW_K @>{\deg_K}>> \ZZ_p,  \end{CD} \]
where the first down-arrow is the norm map and the second is the multiplication by $n$.
Therefore, since ${\tilde L}/K$ contains the constant $\ZZ_p$-extension, the exact sequence \eqref{e:weil} implies
\begin{equation}\label{e:fwfclim} \fW_{\tilde L}=\fC_{\tilde L}:=\varprojlim_F\fC_F. \end{equation}

\begin{lemma} \label{l:wtorsion} The group $\fW_{\tilde L}$ has no $p$-torsion. \end{lemma}

\begin{proof} For $F$ a finite extension of $K$ put $F_n:=\FF(A[p^n])F$. It suffices to show that for any $F$ the projective
limit (on $n$) of $\fC_{F_n}$ has no $p$-torsion. Let $\C/\FF_F$ be the curve associated with the function field $F$ (whose field of constants is $\FF_F$), so that
$$\fC_{F_n}=Jac\,\C[p^\infty]\cap Jac\,\C(\FF_{F_n}).$$
Let $t=(t_n)$ be an element in the $p$-torsion of $\varprojlim\fC_{F_n}$: then $t_n\in Jac\,\C[p]$ for all $n$ and hence there is some $m$ such that $t_n\in Jac\,\C(\FF_{F_m})$ for every $n$. But then for $n>m$ the norm map $\fC_{F_n}\rightarrow\fC_{F_{n-1}}$ acts on $t_n$ as multiplication by a power of $p$, hence $t=0$.
\end{proof}

\end{subsubsection}
\end{subsection}

\begin{subsection}{Frobenius Twist of Class Groups}\label{su:ftcg}
Let $\cO$ denote the ring of integers of some finite extension over $\QQ_p$. Recall the algebras $\La_{\cO}(\Gamma)$ and $\La_{\cO}(\tilde \Gamma)$ defined in Subsection \ref{su:iw}.

\subsubsection{The twist matrix} \label{ss:twist}
Since $A$ is defined over $\FF$, $A[p^{\infty}]$ is actually a $\Gal(\bar\FF/\FF)$-module. The latter group is topologically generated by the Frobenius substitution $\Fr_q$.
After the choice of an isomorphism $A[p^{\infty}]\simeq(\QQ_p/\ZZ_p)^g$, the action of $\Fr_q$ becomes that of a $g\times g$ matrix $\bo u$, called the twist matrix (see page 216 of \cite{mazur72} for a more detailed discussion). Note that the actions of $\Fr_q$ on $A[p^\infty]$ coincides with that of the Frobenius endomorphism $\tF_q$. By \cite[Corollary 4.37]{mazur72}, the eigenvalues of $\tF_q$ of $A$ are
$$\alpha_1,...,\alpha_g,\, \beta_1:=q/\alpha_1,...,\beta_g:=q/\alpha_g,$$
where $\alpha_1$,...,$\alpha_g$ are the eigenvalues of $\bo u$, counted with multiplicities. Assume that $\cO$ contains $\alpha_i$: then
\begin{equation} \label{e:alphabeta} \alpha_i\in\cO^{\times}\text{, }\beta_i\in q\cO\,. \end{equation}

\begin{lemma}\label{l:ses}
The twist matrix $\bo u$ is semi-simple. If $A$ is simple, then $\bo u$ has distinct eigenvalues.
\end{lemma}

\begin{proof} It is sufficient to assume that $A$ is simple, whence $\QQ\otimes \mathcal{E}$ is a field by Lemma \ref{l:silly}.
Then by \cite[Theorem 2(c)]{Ta66}, the algebra
$\QQ\otimes \mathcal{E}$ equals $\QQ(\tF_q)$, and is a filed extension of $\QQ$ of degree $2g$.
Therefore, $\alpha_1,...,\alpha_g, \beta_1,...,\beta_g$ are exactly the roots of the minimal polynomial of $\tF_q$ over $\QQ$.
In particular, they are distinct.
\end{proof}

\subsubsection{The Frobenius action on $A[p^\infty]^\vee$}\label{su:lambdai}
The action of $\Fr_q$ on $A[p^{\infty}]^\vee\simeq \ZZ_p^g$ is through the inverse matrix $\bo u^{-1}$. Suppose $\cO$ contains all $\alpha_1,...,\alpha_g$.
Let $\lambda_i$ denotes the continuous homomorphism $\tilde\Gamma \rightarrow \Gal(K(A[p^{\infty}])/K)\rightarrow \cO^{\times}$ such that
\begin{equation}\label{e:lambdai} \lambda_i(\Fr_q):=\alpha_i^{-1} \end{equation}
Let $\cO(\lambda_i)$ be the twist of $\cO$ by $\lambda_i$ defined in \S\ref{su:twistmodule} (see \eqref{e:twistbyphi} and the lines just after it). Let $H$ be the group in \eqref{e:prod2}.  Denote $\psi_i:=\lambda_{i|H}$.
Note that Proposition \ref{p:selclassg} implies that $\fW_{\tilde L}$ is finitely generated over $\La(\Gamma_{ar})$.

\begin{proposition}\label{p:cs} Assume that $\fW_{\tilde L}$ is torsion over $\La(\Gamma_{ar})$ and that $\cO$ contains $\alpha_1,...,\alpha_g$ as well as $\psi(h)$ for all $\psi\in H^\vee$, $h\in H$. Then for every $\psi\in H^\vee$ there is a pseudo-isomorphism of $\Lambda_{\cO}(\Gamma_{ar})$-modules
$$\big(\cO\F X_p(A/\tilde L)\big)^{(\psi)}\sim \bigoplus_{i=1}^g \cO\fW_{\tilde L}^{(\psi\psi_i^{-1})}\otimes_{\cO} \cO(\lambda_i)\,.$$
Moreover, we have
$$\cO\F X_p(A/L_{ar}) \sim \bigoplus_{i=1}^g \cO\fW_{\tilde L}^{(\psi_i^{-1})}\otimes_{\cO} \cO(\lambda_i)\,.$$
\end{proposition}

\begin{proof} Since $\Fr_q$ topologically generates $\Gal(K(A[p^\infty])/K)$, Lemma \ref{l:ses} yields an exact sequence of $\cO[[\Gal(K(A[p^{\infty}])/K)]]$-modules
\begin{equation}\label{e:vax} 0\lr \bigoplus_{i=1}^g \cO(\lambda_i) \lr \cO A[p^{\infty}]^\vee \lr Q \lr 0   \end{equation}
with $p^m Q=0$ for some $m$. Lemma \ref{l:wtorsion} implies that $\fW_{\tilde L}$ is a flat $\ZZ_p$-module and is annihilated by some $f\in\La(\Gamma_{ar})$ coprime with $p\,$: by Lemma \ref{l:psn}, it follows that the module $\fW_{\tilde L}\otimes_{\ZZ_p} Q$ is pseudo-null over $\La_{\cO}(\Gamma_{ar})$. Hence \eqref{e:vax} implies that $\fW_{\tilde L}\otimes_{\ZZ_p}\cO A[p^{\infty}]^\vee$ is pseudo-isomorphic to $\bigoplus_{i=1}^g \fW_{\tilde L}\otimes_{\ZZ_p} \cO(\lambda_i)$. Since the group $H$ is of order prime to $p$,
$$\bigoplus_{i=1}^g \fW_{\tilde L}\otimes_{\ZZ_p} \cO(\lambda_i)=\bigoplus_{i=1}^g \bigoplus_{\varphi\in H^\vee}\cO\fW_{\tilde L}^{(\varphi)}\otimes_{\cO} \cO(\lambda_i)\,.$$
The first statement then follows from Proposition \ref{p:selclassg}.
To prove the second claim, we apply the inflation-restriction sequence and the restriction-corestriction formula
to obtain
$$\Sel_{p^\infty}(A/L_{ar})=\Sel_{p^{\infty}}(A/{\tilde L})^{H},$$
also by the fact that the order of $H$ is prime to $p$.
Therefore, if $h$ is a generator of $H$ then
$$\cO\F X_p(A/L_{ar})=\cO\F X_p(A/\tilde L)/(1-h)\cO\F X_p(A/\tilde L)$$
which, according to the above argument, is pseudo-isomorphic to $\bigoplus \fW_{\tilde L}^{(\psi_i^{-1})}\otimes_{\ZZ_p} \cO(\lambda_i).$
\end{proof}

\end{subsection}
\end{section}


\begin{section}{The Iwasawa Main Conjecture for constant abelian varieties} \label{s:imcconstant}

In this section, we keep the notation of \S\ref{su:sgcg} and \S\ref{su:ftcg}.
We shall identify $\Gal(K(A[p^\infty])/K)$ with $\Gal(\FF(A[p^\infty])/\FF)$ and consider the Frobenius substitution $\Fr_q$ as an element of it.

\begin{subsection}{The Stickelberger element and divisor class groups}\label{su:sc}
For a topological ring $R$, the ring of formal power series $R[[u]]$ is endowed with the topology coming from identification with an infinite product of copies of $R$. Any continuous ring homomorphism $\phi\colon R\rightarrow R'$ extends to a continuous ring homomorphism $R[[u]]\rightarrow R'[[u]]$
by $u\mapsto u$.

\subsubsection{The Stickelberger element} \label{ss:stickelb}
Let $M/K$ be an abelian Galois extension (of finite or infinite degree) unramified outside $S$; in case $S=\emptyset$, we furthermore stipulate that $M$ is a subfield of $K\bar\FF$, so that $\Gal(M/K)$ is generated by $\Fr_q$ (note that if $S=\emptyset$ then $L=\Kpinf$, $\tilde L=K(A[p^\infty])$, because non-arithmetic abelian totally unramified extensions of $K$ have Galois group a factor of the finite group $\fC_K$).

For any place $v$ outside $S$ let the symbol $[v]_M\in\Gal(M/K)$ denote the (arithmetic) Frobenius element at $v$.
Set $\nabla_S(u):=1-\Fr_q \cdot u $ if $S=\emptyset$ and $\nabla_S(u):=1$ if $S\not=\emptyset$. Since there are only finitely many places with degree bounded by a given positive integer, we can express the infinite product
\begin{equation}\label{e:Thetaprod}  \Theta_{M,S}(u):=\nabla_S(u)\cdot \prod_{v\not\in S} (1-[v]_M\cdot u^{\deg(v)})^{-1} \end{equation}
as a formal power series in $\ZZ[\Gal(M/K)][[u]]\subset \ZZ_p[[\Gal(M/K)]][[u]]$.
For any extension $M/M'$, the natural map $\Gal(M/K)\rightarrow \Gal(M'/K)$ induces a continuous $\ZZ_p$-algebra homomorphism
$$p_{M/M'}:\ZZ_p[[ \Gal(M/K)]][[u]]\rightarrow\ZZ_p[[\Gal(M'/K)]][[u]]$$
by $u\mapsto u$.  We have
\begin{equation} \label{e:pmmtheta}
p_{M/M'}(\Theta_{M,S}(u))=\Theta_{M',S}(u).
\end{equation}

Let $\omega\colon{\Gal(M/K)}\rightarrow \CC^{\times}$ be a continuous character. Denote by $S_\omega\subseteq S$ its ramification locus and $K_\omega$ the fixed field of $\Ker(\omega)$. For $v$ a place of $K$ set $q_v:=q^{\deg(v)}$. Define
\begin{equation} \label{e:L(omega)} L(\omega,s):=\prod_{v\notin S_\omega}(1-\omega([v]_{K_\omega}) q_v^{-s})^{-1}. \end{equation}

From now on, we fix embeddings $\bar{\QQ}\subset \bar{\QQ}_p$ and $\bar{\QQ}\subset \CC$. Thus, if $\cO$ contains the image of $\omega$,
which is finite, then $\omega$ can be viewed as a character $\Gal(M/K)\rightarrow \cO^{\times}$ that extends $\cO$-linearly to a ring homomorphism
$\omega\colon\cO[[  \Gal(M/K)]][[u]]\rightarrow \cO[[u]]$, by $u\mapsto u$ again.
Applying $\omega$ to $\nabla_S$, we get a polynomial in $\cO[u]$. Define
\begin{equation} \label{e:L(S,omega)} L^*_S(\omega,s):=\omega(\nabla_S)(q^{-s})\cdot \prod_{v\not\in S} (1-\omega([v]_M)\,q_v^{-s})^{-1}. \end{equation}

If $\alpha\in\bar{\QQ}$ and is contained in $\cO$, then we can ask if $\Theta_{M,S}(\alpha)$ converges and, if it does, what is the value of $\omega(\Theta_{M,S}(\alpha))$.
Recall that for any topological ring $R$ the Tate algebra $R\langle u\rangle$ consists of those power series in $R[[u]]$ whose coefficients tend to 0. The algebra $\ZZ_p[[\Gal(M/K)]]$ is the inverse limit of $\ZZ_p[\Gal(F/K)]$ as $F$ varies among all finite subextensions: therefore a power series $f\in\ZZ_p[[\Gal(M/K)]][[u]]$ belongs to the Tate algebra if and only if $p_{M/F}(f)$ is in $\ZZ_p[\Gal(F/K)][u]$ for all $F$.

\begin{proposition} \label{p:Thetaconv} Assume the above notation.
The power series $\Theta_{M,S}(u)$ belongs to $\ZZ_p[[\Gal(M/K)]]\langle u\rangle$, and hence $\Theta_{M,S}(\alpha)$
converges for every $\alpha\in\bar{\QQ}\cap \cO$. If furthermore $\alpha=q^{-s_0}$ for some $s_0\in\CC$,
then
\begin{equation} \label{e:evalTheta}  \omega\big(\Theta_{M,S}(\alpha)\big)=L_S^*(\omega,s_0)\,. \end{equation}
\end{proposition}

\begin{proof} Choose a non-empty finite set $T$ of places of $K$ such that $S\cap T=\emptyset$ and define
$$\Theta_{M,S,T}(u):=\Theta_{M,S}(u) \cdot f_T(u),$$
with $f_T(u):=\prod_{v\in T} (1-q_v[v]_M\cdot u^{\deg(v)})$. Then $\omega(f_T)(q^{-s})=\prod_{v\in T} (1-\omega([v]_M)\cdot q_v^{1-s})$. Define
$$ L^*_{S,T}(\omega,s):= L^*_S(\omega,s)\cdot \omega(f_T)(q^{-s}).$$
By the above definitions and \eqref{e:Thetaprod}, \eqref{e:L(S,omega)} we get $\omega(\Theta_{M,S,T})(q^{-s})=L^*_{S,T}(\omega,s)$. A result of Weil \cite[VII, Theorems 4 and 6]{We73} \footnote{Amusingly, the proof in Weil's book is essentially due to Tate, while Weil's geometric argument is explained by Tate in \cite[V, Th\'eor\`eme 2.5]{t84}.} yields
$$L_{S,T}^*(\omega,s)=P_{\omega,S,T}(q^{-s}),$$
where $P_{\omega,S,T}(u)$ is a polynomial in $\cO[u]$. Note that if $\omega$ factors through $p_{M/F}$, then $P_{\omega,S,T}(q^{-s})$ also equals $\omega(\Theta_{F,S,T})(q^{-s})$, by \eqref{e:pmmtheta}.
This implies that if $F/K$ is a finite extension, then $\Theta_{F,S,T}(u)$ belongs to $ \ZZ[\Gal(F/K)][u]$ (since $\omega(\Theta_{F,S,T})(u)$ is a polynomial for arbitrary $\omega\in\Gal(F/K)^\vee$).

Therefore, $\Theta_{M,S,T}\in\ZZ_p[[\Gal(M/K)]]\langle u\rangle$. To complete the proof of the first statement it suffices to observe that
$f_T(u)$ is a unit in $\ZZ_p[[\Gal(M/K)]]\langle u\rangle$.

As for \eqref{e:evalTheta}, choose $F$ such that $\omega$ factors through $p_{M/F}$. Then $\omega(\Theta_{M,S,T}(\alpha))=\omega(\Theta_{M,F,T}(\alpha))$.
But since $\Theta_{M,F,T}(u)$ is a polynomial, we have $\omega(\Theta_{M,S,T}(\alpha))=P_{\omega,S,T}(q^{-s_0})=L_{S,T}^*(\omega,s_0)$. To conclude note that $L_S^*(\omega,s_0)$ is well defined because $\omega(f_T)(q^{-s_0})\not=0$, since $\alpha\in\cO$.

\end{proof}

Define the Stickelberger element
\begin{equation} \label{e:thetaLST} \theta_{M,S}:=\Theta_{M,S}(1)\in\ZZ_p[[\Gal(M/K)]]\,.   \end{equation}

\begin{lemma} \label{l:thetanot0} If $\Kpinf\subseteq M$, then $\theta_{M,S}\neq0\,.$
\end{lemma}

\begin{proof}
By \eqref{e:evalTheta}, it suffices to show $L_S^*(\omega,0)\neq0$ for some $\omega\in \Gal(\Kpinf/K)^\vee\subset \Gal(M/K)^\vee$. Observe that for such an $\omega$, the $L$-function $L(\omega, s)$ is just a twist of the Dedekind zeta function $\zeta_K(s)$: hence, letting $\varepsilon:=\omega(\Fr_q)$ and $\varepsilon_v=\varepsilon^{\deg(v)}$, we have
$$L_S^*(\omega,s)=\omega(\nabla_S)(q^{-s}) \prod_{v\in S} (1-\varepsilon_vq_v^{-s})\cdot \frac{P_K(\varepsilon q^{-s})}{(1-\varepsilon q^{-s})(1-\varepsilon q^{1-s})}$$
since $(1-q^{-s})(1-q^{1-s})\zeta_K(s)=P_K(q^{-s})$ for some polynomial $P_K$. This proves our claim, because $P_K(\varepsilon)\neq0$ by the Riemann hypothesis over function fields and if we choose $\omega$ of order sufficiently high then
$\varepsilon\neq1$ and $\varepsilon_v\neq1$ for all $v\in S$.
\end{proof}



\subsubsection{Relation with class groups}
In addition to the maps defined in (H1), (H2) of \S\ref{su:iw}, we will use the following morphism:
\begin{enumerate}
\item[(H3)] Since ${\tilde \Gamma}= \Gamma_{ar}\times H$, any character $\psi\colon H\lr\cO^\times$ can be extended to a homomorphism $\psi\colon\La_\cO(\tilde \Gamma) \lr \La_\cO(\Gamma_{ar})\,$, via $(h,g)\mapsto\psi(h)g$.
\end{enumerate}

For each $\psi\in H^\vee$ consider
\begin{equation} \theta_{\psi,L_{ar},S}^\sharp:=\psi(\theta_{\tilde L,S}^\sharp)=\psi\big(\Theta_{\tilde L,S}(1)^\sharp\big)\in \La_{\cO}(\Gamma_{ar}). \end{equation}
The following theorem was proved in \cite{crw87}; a simplified proof (avoiding the use of crystalline cohomology) has recently been given in \cite{blt09}.

\begin{mytheorem}\label{t:class} Let notation be as above and assume that $\psi\in H^\vee$ with $\psi(h)\in\cO$ for every $h\in H$. Then $\cO \fW_{\tilde L}^{(\psi)}$ is a finitely generated
$\La_{\cO}(\Gamma_{ar})$-module and
$$\chi_{\cO,\Gamma_{ar}}(\cO \fW_{\tilde L}^{(\psi)})= (\theta_{\psi,L_{ar},S}^\sharp)\,.$$
\end{mytheorem}

From $\omega\in\Gamma_{ar}^\vee$ one obtains a character $\omega\times\psi$ via \eqref{e:prod2}. Then $\omega(\theta_{\psi,L_{ar},S})=(\omega\times\psi) (\theta_{\tilde L,S})= L_S^*(\omega\times\psi,0)$ and the proof of Lemma \ref{l:thetanot0} is easily adapted to show that $\theta_{\psi,L_{ar},S}\neq0$. Thus Theorem \ref{t:class} implies $\fW_{\tilde L}$ is torsion over $\Lambda(\Gamma_{ar})$. By Proposition \ref{p:cs} we get

\begin{corollary} Both $\F X_p(A/\tilde L)$ and $\F X_p(A/L_{ar})$ are torsion $\La_\cO(\Gamma_{ar})$-modules. \end{corollary}

\end{subsection}

\begin{subsection}{Frobenius Twist of Stickelberger Elements} \label{su:cis}

In this section, we compute the characteristic ideal of $X_p(A/L_{ar})$. As in \S\ref{su:lambdai}, we assume that the ring $\cO$ contains all $\alpha_1,...,\alpha_g$ as well as $\psi(h)$ for all $\psi\in H^\vee$, $h\in H$.

\begin{proposition}\label{p:cs1}
Let $\psi\in H^\vee$ and let $\lambda_i$ and $\psi_i$ be as in {\em{\S\ref{su:lambdai}}}. Then
$$\chi_{\cO,\Gamma_{ar}}\big((\cO\F X_p(A/\tilde L))^{(\psi)}\big) = \prod_{i=1}^g \big(\lambda_i^* (\theta_{\psi\psi_i^{-1},L_{ar},S}^\sharp)\big) $$
and
$$\chi_{\cO,\Gamma_{ar}}\big(\cO\F X_p(A/L_{ar})\big) = \prod_{i=1}^g \big(\lambda_i^* (\theta_{\psi_i^{-1},L_{ar},S}^\sharp)\big)\,.$$
\end{proposition}

\begin{proof} The statement follows from Proposition \ref{p:cs}, Lemma \ref{l:phi[]chi} and Theorem \ref{t:class}. \end{proof}

\begin{subsubsection}{The Stickelberger element for $A$}  \label{ss:stickA} Let $M/K$ be an extension as in \S\ref{ss:stickelb}. Define
\begin{equation}\label{e:thetasi}
\theta_{A,M,S,i}^+:=\Theta_{M,S}(\alpha_i^{-1})^\sharp  \in \cO[[\Gal(M/K)]],\,\text{for}\, i=1,...,g.
\end{equation}
This makes sense by Proposition \ref{p:Thetaconv}, because $\alpha_i$ is a unit in $\cO$. Put
\begin{equation}\label{e:theta+} \theta_{A,M,S}^+:=\prod_{i=1}^g \theta_{A,M,S,i}^+\,. \end{equation}
Note that $\theta_{A,L,S}^+\in\La$ and $\theta_{A,L_{ar},S}^+\in\La(\Gamma_{ar})$, since the set $\{\alpha_1,...,\alpha_g\}$ is stable under the action of $\Gal({\bar \QQ_p}/\QQ_p)$. Define
\begin{equation}\label{e:plfunction}
\mathcal{L}_{A,M}:=\theta_{A,M,S}^+\cdot(\theta_{A,L,S}^+)^\sharp.
\end{equation}

\begin{proposition}\label{p:cs2} We have
$$\chi(X_p(A/L_{ar}))=(\mathcal{L}_{A,L_{ar}}).$$
\end{proposition}

\begin{proof} Since $A$ and $A^t$ are isogenous, they share the same twist matrix. Thus, in view of Proposition \ref{p:xffsharp} and Proposition \ref{p:cs1}, we only need to prove the equality
$$\lambda_i^*(\theta_{\psi_i^{-1},L_{ar},S}^\sharp)=\theta_{A,L_{ar},S,i}^+\,.$$
But this is just a matter of unwinding definitions. The composition $\lambda_i^*\circ\psi_i^{-1}\circ\cdot^\sharp \colon\tilde\Gamma\lr \La_{\cO}(\Gamma_{ar})$ sends $\gamma=(\gamma_H,\gamma_{\Gamma_{ar}})$ to $\lambda_i(\gamma)\gamma_{\Gamma_{ar}}^{-1}$. In particular,
$$(\lambda_i^*\circ\psi_i^{-1}\circ\cdot^\sharp)(1-[v]_{\tilde L}\cdot u^{\deg(v)})=1-\lambda_i([v]_{\tilde L})[v]_{L_{ar}}^{-1} u^{\deg(v)}=1-[v]_{L_{ar}}^{-1}(\alpha_i^{-1} u)^{\deg(v)}\,,$$
which implies
$$\lambda_i^* \big(\psi_i^{-1}(\Theta_{\tilde L,S}(1)^\sharp)\big)= \Theta_{L_{ar},S}(\alpha_i^{-1})^\sharp =\theta_{A,L_{ar},S,i}^+\,.$$
\end{proof}

\end{subsubsection}
\end{subsection}

\begin{subsection}{$p$-adic interpolation of the $L$-function} \label{s:interpol}
Let $\alpha_i,\beta_i$ be the algebraic integers introduced in \S\ref{ss:twist}.
Recall that the $L$-function of $A$ is (see \cite[Corollary 4.37]{mazur72})
$$L(A,s):=\prod_v \prod_{i=1}^g \big(1-\beta_{i,v}q_v^{-s}\big)^{-1}\big(1-\alpha_{i,v}q_v^{-s}\big)^{-1},$$
where $\alpha_{i,v}:=\alpha_i^{\deg(v)}$ and $\beta_{i,v}:=\beta_i^{\deg(v)}=q_v\alpha_{i,v}^{-1}$.
For a continuous character $\omega\colon\Gamma\rightarrow \boldsymbol\CC^\times$, the twisted $L$-function is then
$$L(A,\omega,s):=\prod_{v\notin S_{\omega}}\prod_{i=1}^g \big(1-\omega([v]_L)\beta_{i,v}q_v^{-s}\big)^{-1}\big(1-\omega([v]_L)\alpha_{i,v}q_v^{-s}\big)^{-1} .$$

Before stating the interpolation formula relating $\mathcal{L}_{A,L}$ with $L_S(A,\omega,1)$, we need to introduce some notation.
 Put
$$\Xi_{S,\omega}:=\prod_{v\in S-S_\omega}\prod_{i=1}^g(1-\omega([v]_{K_\omega})^{-1}\alpha_{i,v}^{-1}) (1-\omega([v]_{K_\omega})\alpha_{i,v}^{-1})
\in\bar{\QQ}$$
 Also, set
$$\nabla_{A,S}:=\prod_{i=1}^g \nabla_S(\alpha_i^{-1})\nabla_S(\alpha_i^{-1})^\sharp \in\QQ[\Gamma]\,.$$

Denote by $\kappa$ the genus of $K$ (that is, the genus of the corresponding curve $C/\FF$) and by $d_\omega$ the degree of the conductor of $\omega$. Fix an additive character $\Psi\colon{\bf A}_K/K\rightarrow\Bmu_{p^\infty}$ on the adele classes of $K$ and let $a=(a_v)$ be a differental idele attached to $\Psi$ (\cite[p.~113]{We73}) and set $b_\omega=(b_{\omega,v})$ with $b_{\omega,v}=a_v\pi_v^{d_{\omega}}$. For every place $v$, let $\alpha_v$ be the self-dual Haar measure on $K_v$ with reference to $\Psi_v$. Then define the Gauss sum
$$\tau(\omega):=\begin{cases}
\displaystyle 1 & \text{if } \omega \text{ factors through } \Gal(K\bar\FF/K) \\ {}\\
\displaystyle \prod_{v\in S_\omega}\frac{1}{|b_{\omega,v}|_v^{1/2}}\int_{O_v^\times}\omega_v(x)^{-1}\Psi_v(b_{\omega,v}^{-1}x)d\alpha_v(x) & \text{otherwise} \end{cases}$$
where $O_v$ is the ring of integers of $K_v$.

\begin{mytheorem} \label{t:interpo} The element $\mathcal{L}_{A,L}$ interpolates the $L$-function of $A$: for any continuous character $\omega\colon\Gamma\rightarrow \CC^\times$ we have
\begin{equation} \label{e:interpo}
\omega(\mathcal{L}_{A,L})=\omega(b)^{-g}\tau(\omega^{-1})^g\big(q^{g/2}\prod_{i=1}^g\alpha_i^{-1}\big)^{2\kappa-2+d_\omega}\,\omega(\nabla_{A,S})\,\Xi_{S,\omega}\, L(A,\omega,1)\,.
\end{equation}
\end{mytheorem}

\begin{proof} By definition, recalling \eqref{e:thetasi}, we have
\begin{equation} \label{e:thetaAprod}
\mathcal{L}_{A,L}=\prod_{i=1}^g\theta_{A,L,S,i}^+(\theta_{A,L,S,i}^+)^\sharp= \prod_{i=1}^g\Theta_{L,S}(\alpha_i^{-1})^\sharp \, \Theta_{L,S}(\alpha_i^{-1})\,.
\end{equation}
We fix an embedding of $\{\alpha_i\mid i=1,...,g\}$ into $\CC$
and for every $i$ fix $s_i\in\CC$ such that $\alpha_i=q^{s_i}$, so that $\beta_i=q^{1-s_i}$. The equality \eqref{e:evalTheta} yields
$$\omega\big(\Theta_{L,S}(\alpha_i^{-1})\big)=L_S^*(\omega,s_i)$$
and
$$\omega\big(\Theta_{L,S}(\alpha_i^{-1})^\sharp\big)=L_S^*(\omega^{-1},s_i)\,,$$
because $\omega(\lambda^\sharp)=\omega^{-1}(\lambda)$ for any $\lambda\in\La$.
It is well known that the $L$-function satisfies the functional equation
\begin{equation}\label{e:functeqL} L(\omega^{-1},s)=\omega(b)^{-1}\tau(\omega^{-1})q^{(\frac{1}{2}-s)(2\kappa-2+d_\omega)}L(\omega,1-s) \end{equation}
(see e.g.~\cite[VII, Theorems 4 and 6]{We73}).
Formula \eqref{e:functeqL} implies
$$L(\omega^{-1},s_i)=\omega(b)^{-1}\tau(\omega^{-1})(q^{1/2}\alpha_i^{-1})^{2\kappa-2+d_\omega}  \cdot L(\omega,1-s_i)\,.$$
Putting everything together, we have obtained
$$\omega(\mathcal{L}_{A,L})=\omega(b)^{-g}\tau(\omega^{-1})^g\big(q^{g/2}\prod_{i=1}^g\alpha_i^{-1}\big)^{2\kappa-2+d_\omega}\,\omega(\nabla_{A,S})\,\Xi_S(\omega)\prod_{i=1}^g L(\omega,s_i)\,L(\omega,1-s_i)\,.$$
But it follows directly from the definition that
$$ L(A,\omega,1)=\prod_{i=1}^g L(\omega,s_i)\,L(\omega,1-s_i) .$$
\end{proof}



\begin{subsection}{Proof of the Main Conjecture} \label{su:MCconstant}

Proposition \ref{p:cs2} proves the Iwasawa Main Conjecture for $A/K$ a constant abelian variety when $L=L_{ar}$. To deal with the general case, we apply the following theorem.
We write $\chi_M$, $\chi_{M'}$ for characteristic ideals in $\ZZ_p[[\Gal(M/K)]]$,
$\ZZ_p[[\Gal(M'/K)]]$ respectively.

\begin{mytheorem}[Tan] \label{t:descent} Let $M/K$ be a $\ZZ_p^e$-extension, ramified at finitely many places, and $M'/K$ a $\ZZ_p^{e-1}$-subextension, $e\geq2$. Define $\vartheta_{M/M'}, \varrho_{M/M'}\in\ZZ_p[[\Gal(M'/K)]]$ by
\begin{equation} \label{e:thetaM/M'} \vartheta_{M/M'}:=\prod_{v\in S_{M/M'}}\prod_{i=1}^g(1-\alpha_{i,v}^{-1}[v]_{M'})(1-\alpha_{i,v}^{-1}[v]_{M'}^{-1}) \end{equation}
{\em{(}}where $ S_{M/M'}$ is the set of places ramified in $M/K$ but not in $M'/K${\em{)}} and
\begin{equation} \label{e:rhoM/M'} \varrho_{M/M'}:=\begin{cases}
\displaystyle \prod_{i=1}^g(1-\alpha_i^{-1}\Fr_q)(1-\alpha_i^{-1}\Fr_q^{-1}) & \text{when } M'=\Kpinf \\
1 & \text{otherwise}.  \end{cases}\end{equation}
Then
\begin{equation} \label{e:descent}
\vartheta_{M/M'}\cdot\chi_{M'}(X_p(A/M'))=\varrho_{M/M'}\cdot p_{M/M'}\big(\chi_M(X_p(A/M))\big)\,.
\end{equation}
\end{mytheorem}

This is just a special case of \cite[Theorem 1]{tan10b}: since $A$ has good reduction everywhere, the only terms appearing in $\prod_v\vartheta_v$ of {\em loc.~cit.} are the ones coming from changes in the set of ramified places, i.e., $\vartheta_{M/M'}$ as defined in \eqref{e:thetaM/M'}.

\begin{mytheorem} \label{t:imc3constant} Let $A$ be a constant ordinary abelian variety over the function field $K$. For any $\ZZ_p^d$-extension $L/K$ unramified outside a finite set of places, we have
\begin{equation} \label{e:imc3} \chi(X_p(A/L))=(\mathcal{L}_{A,L}). \end{equation}
\end{mytheorem}

\begin{proof}
We only need to consider the $L \neq L_{ar}$ case. In this case, $S\not=\emptyset$, and hence $p_{L_{ar}/L}(\mathcal{L}_{A,L_{ar}})=\mathcal{L}_{A,L}\,$.
Apply Theorem \ref{t:descent} with $M=L_{ar}$ and $M'=L$. Then obviously $S_{M/M'}=\emptyset$, and hence $\vartheta_{M/M'}=1$.
Also, $\varrho_{L_{ar}/L}=1$ since $\Kpinf$ is not contained in $L$. Therefore, \eqref{e:descent} implies the desired equality
$$\chi(X_p(A/L))=(\mathcal{L}_{A,L}).$$
\end{proof}

An immediate consequence is the following criterion, which might turn out useful for computationally verifying when $X_p(A/L)$ is $\La$-torsion.

\begin{corollary} The Selmer group $\Sel_{p^\infty}(A/L)$ has positive $\La$-corank if and only if $\Theta(u)$ vanishes at some $\alpha_i^{-1}$.
\end{corollary}

\end{subsection}

\subsection{Comparison with \cite{LLTT2}} \label{su:2dreams}
We compare our Main Conjecture and the one in \cite{LLTT2} in the case where $L=\Kpinf$.

\begin{mytheorem} Let $A/K$ be a constant ordinary abelian variety. Then {\em{Theorem \ref{t:imc3constant}}} is equivalent to the Iwasawa Main Conjecture proved in \cite{LLTT2}.
\end{mytheorem}

\begin{proof}

In this case $S=\emptyset$ and we have $\Xi_{S,\omega}=1$ and $d_\omega=0$ for all $\omega$, while
$$\nabla_{A,S}=\prod_{i=1}^g(1-\alpha_i^{-1}\Fr_q^{-1})(1-\alpha_i^{-1}\Fr_q)\in \La.$$
Define
\begin{equation} \label{e:cLtilde}
\tilde{\cL}_{A, \Kpinf}:=\frac{1}{q^{g(\kappa-1)}\nabla_{A,S}}\cdot\big(\prod_{i=1}^g\alpha_i\big)^{2\kappa-2}\cdot\Fr_q^{g(2\kappa-2)}\cdot\mathcal{L}_{A,\Kpinf}\in Q(\La)\,, \end{equation}
where $Q(\La)$ is the fraction field of $\La$. By Theorem \ref{t:interpo} we see that
\begin{equation} \label{e:interpoLtilde} \omega(\tilde{\cL}_{A, \Kpinf})=L(A,\omega,1) \end{equation}
for all characters $\omega$ of $\Gamma$. By \cite[Theorem 1.1, Proposition 2.2.4]{LLTT2}, there is a  generator
$c_{A/\Kpinf}\in \La$ of $\chi(X_p(A/\Kpinf)$ and some $f_{A/\Kpinf}$ and $\star_{A/\Kpinf}$ belonging to $ Q(\La)$ satisfying
$$f_{A/\Kpinf}=\star_{A/\Kpinf}\cdot c_{A/\Kpinf},$$
as well as (since the ramification locus of $\Kpinf/K$ is the empty set)
\begin{equation}\label{e:fak}
\omega(f_{A/\Kpinf})=L(A,\omega,1)
\end{equation}
for all characters $\omega$ of $\Gamma$. If $K$ is the function field of some curve $C/\FF$, then for our $A/\FF$, the abelian scheme
$\mathcal{A}:=A\times_{\FF} C$ is the N$\acute{\text{e}}$ron model of $A\times_{\FF}\Spec K$, and hence the
$\delta$ defined in \cite[(26)]{LLTT2} actually equals $0$, because we can choose $e_1,...,e_g$ to be a basis of $Lie( A)(\FF)$ so that $f_v$ can be taken to be $e$. Thus we have
\begin{equation}\label{e:starak}
\star_{A,K_\infty^{(p)}}= q^{-g(\kappa-1)}\cdot \big(\prod_{i=1}^l(1-\zeta_i^{-1}\Fr_q)(1-\zeta_i^{-1}\Fr_q^{-1})\big)^{-1},
\end{equation}
where $\zeta_1,...\zeta_l$ are eigenvalues of the Galois action of $\Fr_q$ on the Tate module $T_pA(\Kpinf)[p^\infty]$
of $A(\Kpinf)[p^\infty]$.
Now \eqref{e:interpoLtilde} and \eqref{e:fak} together imply that $f_{A/\Kpinf}=\tilde{\cL}_{A, \Kpinf}$.
Since $\Fr_q$ and each $\alpha_i$ belong to $\La^\times$,  in view of
\eqref{e:cLtilde} and \eqref{e:starak}, we only need to show
$$\nabla_{A,S}\cdot \big(\prod_{i=1}^l(1-\zeta_i^{-1}\Fr_q)(1-\zeta_i^{-1}\Fr_q^{-1})\big)^{-1}\in \La^\times.$$
But since $\alpha_1,...,\alpha_g$ are the eigenvalues of $\Fr_q$ acting on the Tate module $T_pA[p^\infty]$ of
$A[p^\infty]$,
we can actually rearrange $\alpha_1,...,\alpha_g$ to have $\alpha_i=\zeta_i$, for $i\leq l$.
Let $\tilde\Gamma$ and $H$ be as in \eqref{e:prod2}.
The maps $\Fr_q\mapsto\alpha_i$ determine homomorphisms of $\tilde\Gamma$ onto some group of $p$-adic units and  the equality $T_pA(\Kpinf)[p^\infty]=(T_pA)^H$ yields that the image of $H$ is non-trivial for $i>l$. Hence $i>l$ implies $\alpha_i\not\equiv1$ mod $p$ and thus $(1-\alpha_i\Fr_q^{\pm1})\in\La^\times$.
\end{proof}

\end{subsection}
\end{section}

\end{document}